\documentclass[twoside]{article}

\usepackage[utf8]{inputenc}
\usepackage{authblk}
\usepackage{lipsum} 

\usepackage[sc]{mathpazo} 
\usepackage[T1]{fontenc} 
\usepackage{microtype} 

\usepackage[a4paper, total={9in, 9in}, hmarginratio=1:1,top=50mm,columnsep=20pt, margin=0.75in]{geometry} 
\usepackage{multicol} 
\usepackage[hang, small,labelfont=bf,up,textfont=it,up]{caption} 
\usepackage{booktabs} 
\usepackage{float} 
\usepackage{hyperref} 

\usepackage{lettrine} 
\usepackage{paralist} 

\usepackage{abstract} 

\usepackage{titlesec} 

\renewcommand{\thesection}{\Roman{section}}
\renewcommand \thesubsection{\Roman{section}.\arabic{subsection}}
\titleformat{\section}[block]{\large\scshape\fontsize{12}{17}\bfseries}{\thesection.}{1em}{} 
\titleformat{\subsection}[block]{\large}{\thesubsection.}{1em}{} 

\usepackage{fancyhdr} 
\usepackage{amsthm}
\usepackage{amsmath}
\usepackage{amsfonts}
\usepackage{amssymb}
\usepackage{verbatim}
\usepackage{enumerate}

\newtheoremstyle{thm}{1.5ex}{1.5ex}{\itshape\rmfamily}{}
{\bfseries\rmfamily}{}{2ex}{}

\newtheoremstyle{rem}{1.3ex}{1.3ex}{\rmfamily}{}
{\itshape}
{} {1.5ex}{}

\theoremstyle{thm}
\newtheorem{theorem}{Theorem}[section]
\newtheorem{lemma}[theorem]{Lemma}

\newtheorem*{Main Theorem}{Main Theorem.}
\newtheorem{corollary}[theorem]{Corollary}
\newtheorem{definition}{Definition}

\theoremstyle{rem}

\title{\vspace{-5mm}\fontsize{24pt}{10pt}\selectfont\textbf{Reconstruction for the Asymmetric Ising Model on Regular Trees}} 

\author[*]{Wenjian Liu}
\author[$\dag$]{Ning Ning}
\affil[*]{\small Department of Mathematics and Computer Science,
Queensborough Community College, City University of New York,
U.S.A., wjliu@qcc.cuny.edu} \affil[$\dag$]{Department of Statistics
and Applied Probability, University of California, Santa Barbara,
U.S.A., ning@pstat.ucsb.edu}

\begin{document}

\date{} 
\maketitle 

\thispagestyle{fancy} 


\begin{multicols}{2} 

\textbf{Keywords:} Asymmetric binary channel; Kesten-Stigum
reconstruction bound; Markov random
fields on trees\\

\textbf{\large Abstract}

It is known that the Kesten-Stigum reconstruction bound is tight for
roughly symmetric binary channels. In this paper, we will adopt a
refined analysis of moment recursion on a weighted version of the
magnetization, which is engaged in~\cite{S} to handle the symmetric
Potts model, and establish the critical condition of the asymmetric
Ising model to make Kesten-Stigum bound the reconstruction threshold
on regular $d$-ary trees.

\section{Introduction}
\subsection{Basic definitions} 
We start with the following broadcasting process that stands as a
discrete, irreducible, aperiodic, and reversible Markov chain. Let
$\mathbb{T}=(\mathbb{V}, \mathbb{E}, \rho)$ be a tree with nodes
$\mathbb{V}$, edges $\mathbb{E}$ and root $\rho\in \mathbb{V}$. Each
edge of the tree acts as a channel on a finite characters set
$\mathcal{C}$, whose elements are configurations on $\mathbb{T}$,
denoted by $\sigma$. Next set a probability transition matrix
$\mathbf{M}=(M_{ij})$ as the noisy communication channel on each
edge. The state of the root $\rho$, denoted by $\sigma_\rho$, is
chosen according to an initial distribution $\pi$ on $\mathcal{C}$.
This symbol is then propagated in the tree as follows. For each
vertex $v$ having as a parent $u$, the spin at $v$ is defined
according to the probabilities
\begin{equation}
\mathbf{P}(\sigma_v=j\mid\sigma_u=i)=M_{i j}
\end{equation}
with $i, j \in \mathcal{C}$. Roughly speaking, the problem of
reconstruction is to investigate whether the symbols received at the
vertices of the $n$th generation contain a non-vanishing information
transmitted by the root as $n$ goes to $\infty$. The following is
the formal definition of the reconstruction.
\begin{definition}
The reconstruction problem for the infinite tree $\mathbb{T}$ is
\emph{solvable} if for some $i, j\in\mathcal{C}$,
\begin{equation}
\limsup_{n\to\infty}d_{TV}(\sigma^i(n), \sigma^j(n))>0
\end{equation}
where $d_{TV}$ is the total variation distance. When
the $\limsup$ is $0$ we will say the model has
\emph{non-reconstruction} on $\mathbb{T}$.
\end{definition}

This paper will restrict to regular $d$-ary trees, i.e., the
infinite rooted tree where every vertex has exactly $d$ offspring.
Let $\sigma(n)$ denote the spins at distance $n$ from the root and
let $\sigma^i(n)$ denote $\sigma(n)$ conditioned on $\sigma_\rho =
i$. The objective model taken into account is the asymmetric binary
channel with the configuration set $\mathcal{C}=\{1, 2\}$, whose
transition matrix is of the form
\begin{equation}
\mathbf{M}= \frac{1}{2} \left[\left(
  \begin{array}{cc}
    1+\theta & 1-\theta \\
    1-\theta & 1+\theta \\
  \end{array}
\right) + \Delta\left(
  \begin{array}{cc}
    -1 & 1 \\
    -1 & 1 \\
  \end{array}
\right)\right],
\end{equation}
where $\Delta$ is used to describe the deviation of $\mathbf{M}$
from the symmetric channel and obviously there is a restriction of
$|\theta|+|\Delta|\leq 1$. Actually the process of broadcasting on a
tree with the channels $\mathbf{M}$ corresponds to the ferromagnetic
Ising model with external field on the tree. Furthermore it is
apparent that the second eigenvalue of the channel $\mathbf{M}$ is
$\theta$ which plays a crucial role in the reconstruction problem.

\subsection{Background}
Determining the reconstruction threshold of a Markov random field in
probability, as the interdisciplinary subject, has attracted more
and more attention from probabilists, statistical physicists,
biologists, etc. In fact, the investigation of the reconstruction
problem originated from spin systems in statistical physics by
establishing that the reconstruction threshold happens to be the
threshold for extremality of the infinite-volume Gibbs measure with
free boundary conditions~\cite{GE}. It is shown that the
reconstruction bound determines the efficiency of the Glauber
dynamics on trees and random graphs~\cite{BE,MA,T}, for example, the
mixing time for the Glauber dynamics undergoes a phase transition at
the reconstruction threshold. The reconstruction threshold is also
believed to play an important role in a variety of other contexts,
such as the efficiency of reconstructing phylogenetic ancestors in
evolutionary biology~\cite{MO2}, communication theory in the study
of noisy computation~\cite{EV}, network tomography~\cite{BH} (derive
the link delays in the interior from end-to-end delays in a computer
network), etc.

It is well known that the reconstruction solvability result when
$d|\theta|^2>1$ for any channel~\cite{KS}. Specially for the binary
symmetric channel, it was shown in~\cite{BL} that $d\theta^2>1$ is
not only the sufficient but necessary condition for the
reconstruction solvability, which we refer to as the Kesten-Stigum
bound. As for all other channels, proving non-reconstructibility
turned out to be harder. Although coupling arguments easily yield
nonreconstruction, these arguments are typically not tight.
Mossel~\cite{MO1,MO3} showed that the Kesten-Stigum bound is not the
bound for reconstruction in the binary-asymmetric model with
sufficiently large asymmetry or in the Potts model with sufficiently
many characters, opening a window to exploit the tightness of the
Kesten-Stigum bound.

In~\cite{S}, the Potts model was completely investigated by means of
the recursive structure of the tree, and more importantly, the
author engaged the refined recursive equations of vector-valued
distributions and concentration analyses to confirm much of the
picture predicted by M\'{e}zard and Montanari~\cite{MM}.

But exact thresholds for non-solvability of the asymmetric Ising
model had not been known until~\cite{BO}, in which Borgs et al
displayed a delicate analysis of the moment recursion on a weighted
version of the magnetization, and thus achieved a breakthrough
result. However this conclusion has just established the existence
of the sufficiently small $\Delta$ without estimating the range of
the symmetry bias to keep Kesten-Stigum bound tight.

\subsection{Main results}
Inspired by Sly~\cite{S}'s work, we are able to present the critical
relationship between $\Delta$ and $\theta$ to preserve tightness of
the Kesten-Stigum bound. Since $d\theta^2>1$ always guarantees the
reconstruction, it suffices to consider $1/2\leq d\theta^2\leq1$ in
the following context.
\begin{theorem}
\label{reconstruction} When $\Delta^2>(1-\theta)^2/3$, for every $d$
the Kesten-Stigum bound is not tight. In other words, the
reconstruction problem is solvable for some $\theta$ even if
$d\theta^2<1$.
\end{theorem}

Furthermore with the assistance of the central limit theorem and
gaussian approximation, we figure out the precise condition to keep
the tightness of the Kesten-Stigum bound for fixed $\pi$ and large
$d$.
\begin{theorem}
\label{nonreconstruction} When $\Delta^2<(1-\theta)^2/3$, there
exists a $D=D(\pi)>0$ such that for $d>D$ the Kesten-Stigum bound is
sharp. Furthermore there is non-reconstruction at the Kesten-Stigum
bound, when $d\theta^2=1$.
\end{theorem}

\section{Main ideas of the proof}
\subsection {Notations}
Note first that the stationary distribution $\pi=(\pi_1, \pi_2)$ of
$\mathbf{M}$ is given by
\begin{equation}
\pi_1=\frac{1}{2}-\frac{\Delta}{2(1-\theta)}\quad \text{and}\quad
\pi_2=\frac{1}{2}+\frac{\Delta}{2(1-\theta)},
\end{equation}
and without loss of generality, it is convenient to assume
$\pi_1\geq\pi_2$. Let $u_1,\ldots,u_d$ be the children of $\rho$ and
$\mathbb{T}_v$ be the subtree of descendants of $v\in \mathbb{T}$.
Furthermore, if we set $d(\cdot, \cdot)$ as the graph-metric
distance on $\mathbb{T}$, denote the $n$th level of the tree by
$L(n)=\{v\in\mathbb{V}: d(\rho, v)=n\}$. With the notation above,
let $\sigma(n)$ and $\sigma_j(n)$ denote the spins on $L_n$ and
$L(n)\cap \mathbb{T}_{u_j}$ respectively. For a configuration $A$ on
$L(n)$ define the posterior function by
\begin{equation}
f_n(i, A)=\mathbf{P}(\sigma_\rho=i\mid\sigma(n)=A).
\end{equation}
By the recursive nature of the tree for a configuration $A$ on
$L(n+1) \cap \mathbb{T}_{u_j}$ we can give the equivalent form of
the previous one
\begin{equation}
f_n(i, A)=\mathbf{P}(\sigma_{u_j}=i\mid\sigma_j(n+1)=A).
\end{equation}
Now for $i=1, 2$ and $1\leq j\leq d$ define
\begin{equation*}
X_i=X_i(n)=f_n(i, \sigma(n));\quad Y_j=Y_j(n)=f_n(1,
\sigma_j^1(n+1))
\end{equation*}
and
\begin{equation*}
X^+=X^+(n)=f_n(1, \sigma^1(n));\quad X^-=X^-(n)=f_n(2, \sigma^2(n)).
\end{equation*}
And it is clear that the random variables $\{Y_j\}_{1\leq j\leq d}$
are independent and identical in distribution. Last introduce the
objective quantities in this paper:
\begin{equation}
x_n=\mathbf{E}(X^+(n)-\pi_1)=\mathbf{E}f_n(1,\sigma^1(n))-\pi_1
\end{equation}
and
\begin{equation}
z_n=\mathbf{E}(X^+(n)-\pi_1)^2=\mathbf{E}(f_n(1,\sigma^1(n))-\pi_1)^2.
\end{equation}
Referring to~(\cite{MO1}, Proposition~14), it suffices to
investigate the asymptotic behavior of $x_n$ as $n$ goes to
infinity. Then we can establish the equivalent condition for
non-reconstruction.
\begin{lemma}
The non-reconstruction is equivalent to
$$
\lim_{n\to\infty}x_n=0.
$$
\end{lemma}

\begin{proof}
The maximum-likelihood algorithm, which is the optimal
reconstruction algorithm of $\sigma_\rho$ given $\sigma(n)$, is
successful with probability
\begin{equation}
\Delta_n=\mathbf{E}\max\{X_1(n), X_2(n)\}.
\end{equation}
Therefore it follows immediately the inequality of $x_n+\pi_1\leq
\Delta_n$. On the other side, recalling the assumption of
$\pi_1\geq\pi_2$, we could apply Cauchy-Schwartz inequality, in
tandem with the identity~\eqref{identityxn} to conclude
\begin{equation}
\label{MLE}
\begin{aligned}
\Delta_n&=\pi_1+\mathbf{E}\max\left\{X_1(n)-\pi_1,
X_2(n)-\pi_1\right\}
\\
&\leq\pi_1+\mathbf{E}\max\left\{X_1(n)-\pi_1, X_2(n)-\pi_2\right\}
\\
&=\pi_1+\mathbf{E}|X_1(n)-\pi_1|
\\
&\leq\pi_1+\left(\mathbf{E}(X_1(n)-\pi_1)^2\right)^{1/2}
\\
&\leq\pi_1+\pi_1^{1/2}x_n^{1/2}.
\end{aligned}
\end{equation}
To sum up, we come up with the inequalities
\begin{eqnarray*}
x_n\leq\Delta_n-\pi_1\leq\pi_1^{1/2}x_n^{1/2},
\end{eqnarray*}
implying that $\lim_{n\to\infty}x_n=0$ is equivalent to
$\lim_{n\to\infty}\Delta_n=\pi_1$, which is in turn equivalent to
non-reconstruction~\cite{MO1}.
\end{proof}

\subsection{Preparations}
Before giving the the outline of the proof, it is convenient to
derive some basic identities concerning $x_n$. First we reveal the
relation between the first and second moments of $X^+$.
\begin{lemma}
For any $n\in \mathbb{N}\cup\{0\}$, we have \begin{eqnarray*}
x_n&=&\frac{1}{\pi_1}\mathbf{E}(X_1-\pi_1)^2
\\
&=&\mathbf{E}(X^+(n)-\pi_1)^2+\frac{\pi_2}{\pi_1}\mathbf{E}(X^-(n)-\pi_2)^2
\\
&\geq& z_n \geq 0.
\end{eqnarray*}

\end{lemma}
\begin{proof}By Bayes' rule, we have
\begin{eqnarray}
\mathbf{E}X^+(n)&=&\mathbf{E}f_n(1,\sigma^1(n))\nonumber
\\
&=&\sum_A f_n(1,A)\mathbf{P}(\sigma(n)=A\mid\sigma_\rho=1)\nonumber
\\
&=&\frac{1}{\pi_1}\sum_Af_n(1,A)^2\mathbf{P}(\sigma(n)=A)\nonumber
\\
&=&\frac{1}{\pi_1}\mathbf{E}(X_1^2)
\end{eqnarray}
and similarly,
\begin{equation}
\mathbf{E}X^-=\mathbf{E}f_n\left(2,\sigma^2(n)\right)=\frac{1}{\pi_2}\mathbf{E}(X_2^2).
\end{equation}
Then it follows from the fact of $\mathbf{E}(X_1)=\pi_1$ that
\begin{equation}
\label{identityxn}
x_n=\frac{1}{\pi_1}\left(\mathbf{E}(X_1^2)-\pi_1^2\right)=\frac{1}{\pi_1}\mathbf{E}(X_1-\pi_1)^2.
\end{equation}
Next referring to the identity $X_1(n)+X_2(n)=1$ we obtain
\begin{equation}
x_n=\frac{1}{\pi_1}\mathbf{E}(X_2-\pi_2)^2
=\frac{\pi_2}{\pi_1}(\mathbf{E}X^-(n)-\pi_2).
\end{equation}
Last from~\eqref{identityxn} we present the quantitative relation
between $x_n$ and $z_n$:
\begin{eqnarray*}
x_n&=&\frac{1}{\pi_1}\left[\mathbf{P}(\sigma_\rho=1)\mathbf{E}\big((X_1-\pi_1)^2\mid\sigma_\rho=1\big)\right]
\\
&
&+\frac{1}{\pi_1}\left[\mathbf{P}(\sigma_\rho=2)\mathbf{E}\big((X_2-\pi_2)^2\mid\sigma_\rho=2\big)\right]
\\
&=&\frac{1}{\pi_1}\left[\pi_1\mathbf{E}(X^+(n)-\pi_1)^2+\pi_2\mathbf{E}(X^-(n)-\pi_2)^2\right]
\\
&=&\mathbf{E}(X^+(n)-\pi_1)^2+\frac{\pi_2}{\pi_1}\mathbf{E}(X^-(n)-\pi_2)^2
\\
&\geq&z_n\geq0.
\end{eqnarray*}
\end{proof}

Next with the preceding results, we could evaluate the means and
variances of $Y_j$.
\begin{lemma}
For each $1\leq j\leq d$, we have
$$
\mathbf{E}(Y_j-\pi_1)=\theta x_n \quad\textup{and}\quad
\mathbf{E}(Y_j-\pi_1)^2=\theta z_n+\pi_1(1-\theta)x_n.
$$
\end{lemma}
\begin{proof}
If $\sigma_{u_j}^1=1$, $Y_j$ is distributed according to $X^+(n)$,
while as $1-X^-(n)$ given $\sigma_{u_j}^1=2$. Therefore display our
discussion in virtue of the total probability formula as
\begin{eqnarray*}
\mathbf{E}(Y_j-\pi_1)&=&\mathbf{P}(\sigma_{u_j}^1=1)\mathbf{E}(X^+(n)-\pi_1)
\\
& &+\mathbf{P}(\sigma_{u_j}^1=2)\mathbf{E}(1-X^-(n)-\pi_1)
\\
&=&M_{11}x_n-M_{12}\frac{\pi_1}{\pi_2}x_n
\\
&=&\theta x_n
\end{eqnarray*}
and similarly,
\begin{eqnarray*}
\mathbf{E}(Y_j-\pi_1)^2&=&\mathbf{P}(\sigma_{u_j}^1=1)\mathbf{E}(X^+(n)-\pi_1)^2
\\
& &+\mathbf{P}(\sigma_{u_j}^1=2)\mathbf{E}(1-X^-(n)-\pi_1)^2
\\
&=&M_{11}\mathbf{E}(X^+(n)-\pi_1)^2+M_{12}\mathbf{E}(X^-(n)-\pi_2)^2
\\
&=&M_{11}z_n+M_{12}\frac{\pi_1}{\pi_2}(x_n-z_n)
\\
&=&\theta z_n+\pi_1(1-\theta)x_n.
\end{eqnarray*}
\end{proof}

\section{Moment recursion}
\subsection{Distributional recursion}
It is known that the asymptotic behavior of $x_n$ plays a crucial
rule in determining the reconstruction, however, it is still too
difficult and not necessary to get the explicit expression for
$x_n$. In fact we only need to investigate the recursive formula of
$x_n$, from which it is possible to illustrate the trend of $x_n$ as
$n$ goes to infinity. Thus the key method is to analyze the
recursive relation between $X^+(n)$ and $X^+(n+1)$ by the structure
of the tree. Suppose $A$ is a configuration on $L(n+1)$ and $A_j$
denotes the restriction to $\mathbb{T}_{u_j}\cap L(n+1)$. Then it
can be concluded from the Markov random field property that

\begin{equation}
\label{recursion}
f_{n+1}(1,A)=\frac{N_1}{N_1+N_2},
\end{equation}
where
\begin{eqnarray*}
N_1&=&\pi_1\prod_{j=1}^d\left[\frac{M_{11}}{\pi_1}f_n(1,A_j)+\frac{M_{12}}{\pi_2}f_n(2,A_j)\right]
\\
&=&\pi_1\prod_{j=1}^d\left[1+\frac{\theta}{\pi_1}(f_n(1,A_j)-\pi_1)\right].
\end{eqnarray*}
and
\begin{eqnarray*}
N_2&=&\pi_2\prod_{j=1}^d\left[\frac{M_{21}}{\pi_1}f_n(1,A_j)+\frac{M_{22}}{\pi_2}f_n(2,A_j)\right]
\\
&=&\pi_2\prod_{j=1}^d\left[1-\frac{\theta}{\pi_2}(f_n(1,A_j)-\pi_1)\right]
\end{eqnarray*}
Next conditioning the root to be $1$ and setting $A=\sigma^1(n+1)$
in~\eqref{recursion} give the recursive formula of the random
variable
\begin{equation}
X^+(n+1)=\frac{\pi_1Z_1}{\pi_1Z_1+\pi_2Z_2},
\end{equation}
where
\begin{eqnarray*}
Z_1&=&\prod_{j=1}^d\left[1+\frac{\theta}{\pi_1}(f_n(1,A_j)-\pi_1)\right]
\\
&=&\prod_{j=1}^d\left[1+\frac{\theta}{\pi_1}(Y_j(n)-\pi_1)\right];
\end{eqnarray*}
\begin{eqnarray*}
Z_2&=&\prod_{j=1}^d\left[1-\frac{\theta}{\pi_2}(f_n(1,A_j)-\pi_1)\right]
\\
&=&\prod_{j=1}^d\left[1-\frac{\theta}{\pi_2}(Y_j(n)-\pi_1)\right].
\end{eqnarray*}

\subsection{Main expansion of $x_{n+1}$}
With all preliminary results, we are ready to figure out the
recursion relation of $x_{n+1}$, say, its major expansions, which
would play a crucial rule in the further discussion. As regards
$x_{n+1}$, we could expand it out by virtue of the identity
\begin{equation}
\label{identity}
\frac{a}{s+r}=\frac{a}{s}-\frac{ar}{s^2}+\frac{r^2}{s^2}\frac{a}{s+r}.
\end{equation}
and specifically plugging $a=\pi_1Z_1$, $r=\pi_1Z_1+\pi_2Z_2-1$ and
$s=1$ in \eqref{identity} yields
\begin{equation}
\label{expansion}
\begin{aligned} x_{n+1}&=\mathbf{E}X^+(n+1)-\pi_1
\\
&=\mathbf{E}(\pi_1Z_1)-\mathbf{E}[\pi_1Z_1(\pi_1Z_1+\pi_2Z_2-1)]
\\
&+\mathbf{E}\left[(\pi_1Z_1+\pi_2Z_2-1)^2\frac{\pi_1Z_1}{\pi_1Z_1+\pi_2Z_2}\right]-\pi_1.
\end{aligned}
\end{equation}

In order to estimate terms in~\eqref{expansion}, we adapt Lemma 2.6
in~\cite{S} to our model, and then obtain Taylor series
approximations of means and variances of $Z_i$s.
\begin{lemma}
\label{Taylor} For each positive integer $k$, there exists a
$C=C(\pi, k)$ only depending on $\pi$ and $k$ such that for each
$0\leq k_1, k_2\leq k$,
$$
\mathbf{E}Z_1^{k_1}Z_2^{k_2}\leq C,
$$
\begin{eqnarray*}
&&\left|\mathbf{E}Z_1^{k_1}Z_2^{k_2}-1-d\left\{\mathbf{E}\left[1+\frac{\theta}{\pi_1}(Y_1(n)-\pi_1)\right]^{k_1}\right.\right.
\\
&&\left.\left.\times\left[1-\frac{\theta}{\pi_2}(Y_1(n)-\pi_2)\right]^{k_2}-1\right\}\right|
\leq Cx_n^2
\end{eqnarray*}
and
\begin{eqnarray*}
&&\left|\mathbf{E}Z_1^{k_1}Z_2^{k_2}-1-d\left\{\mathbf{E}\left[1+\frac{\theta}{\pi_1}(Y_1(n)-\pi_1)\right]^{k_1}\right.\right.
\\
&&\left.\times\left[1-\frac{\theta}{\pi_2}(Y_1(n)-\pi_2)\right]^{k_2}-1\right\}-\frac{d(d-1)}{2}
\\
&&\left.\left\{\mathbf{E}\left[1+\frac{\theta}{\pi_1}(Y_1(n)-\pi_1)\right]^{k_1}\left[1-\frac{\theta}{\pi_2}(Y_1(n)-\pi_2)\right]^{k_2}-1\right\}^2\right|
\\
&&\leq Cx_n^3.
\end{eqnarray*}
\end{lemma}

Taken together, plugging all the previous results
in~\eqref{expansion} yields
\begin{equation}
\label{explicit}
\begin{aligned}
x_{n+1}&=\mathbf{E}(\pi_1Z_1)-\mathbf{E}\pi_1Z_1(\pi_1Z_1+\pi_2Z_2-1)
\\
& \quad+\pi_1\mathbf{E}(\pi_1Z_1+\pi_2Z_2-1)^2-\pi_1+S
\\
&=d\theta^2x_n+\frac{1-6\pi_1\pi_2}{\pi_1\pi_2^2}\frac{d(d-1)}{2}\theta^4x_n^2+R+S+T
\end{aligned}
\end{equation}
where $|T|\leq C_T(\pi)x_n^3$ and
\begin{equation}
\label{R} |R|\leq
C_R(\pi)\frac{d(d-1)}{2}|\theta|^5\left|\frac{z_n}{x_n}-\pi_1\right|x_n^2
=O_\pi\left(\left|\frac{z_n}{x_n}-\pi_1\right|x_n^2\right)
\end{equation}
with $C_T(\pi), C_R(\pi)$ constants depending only on $\pi$, and
\begin{equation}
\label{S}
S=\mathbf{E}(\pi_1Z_1+\pi_2Z_2-1)^2\left(\frac{\pi_1Z_1}{\pi_1Z_1+\pi_2Z_2}-\pi_1\right)
\end{equation}
will be handled in the following concentration analysis.

\section{Sufficient condition for the non-tightness of the Kesten-Stigum bound}
\subsection{Estimates of $R$ and $S$} The purpose of the following
lemma is to describe how close the linear term in the recursive
expansion approaches to $x_{n+1}$.
\begin{lemma}
\label{lemLA} For any $\varepsilon>0$, there exists a constant
$\delta=\delta(\pi, \varepsilon)$ such that for all $n$, if
$x_n<\delta$ then
$$
|x_{n+1}-d\theta^2x_n|\leq \varepsilon x_n.
$$
\end{lemma}
\begin{proof}
First note that $Z_1, Z_2\geq0$, and thus $0\leq
\frac{\pi_1Z_1}{\pi_1Z_1+\pi_2Z_2}\leq 1$. Then it is concluded
from~\eqref{expansion} and~\ref{Taylor} that
\begin{equation*}
|x_{n+1}-d\theta^2x_n|\leq Cx_n^2\leq\varepsilon x_n,
\end{equation*}
where $C=C(\pi)$ depends only on $\pi$, the first inequality follows
from the fact of $0\leq z_n\leq x_n$, and the last holds if
$x_n<\delta$ for $\delta=\delta(\pi, \varepsilon)$ small enough.
\end{proof}
Before investigating the concentration, we would like to introduce a
significant lemma showing that $x_n$ does not drop from a very large
value to a very small one.
\begin{lemma}
\label{ndtf} For any fixed $\varrho>0$, assume $|\theta|>\varrho$.
Then there exists a constant $\gamma=\gamma(\pi, \varrho)>0$ such
that for all $n$,
$$
x_{n+1}\geq \gamma x_n.
$$
\end{lemma}
\begin{proof} For a configuration $A=(A_1,\ldots,A_d)$ on
$L(n+1)$ with $A_j$ on $\mathbb{T}_{u_j}\cap L(n+1)$ define
\begin{eqnarray*}
&&f_{n+1}^*(1, A)=\mathbf{P}(\sigma_\rho=1\mid\sigma_1(n+1)=A_1)
\\
&=&\pi_1\frac{\mathbf{P}(\sigma_1(n+1)=A\mid\sigma_\rho=1)}{\mathbf{P}(\sigma_1(n+1)=A)}
\\
&=&\pi_1\frac{M_{11}\mathbf{P}(\sigma_1(n+1)=A\mid\sigma_{u_1}=1)}{\mathbf{P}(\sigma_1(n+1)=A)}
\\
&
&+\pi_1\frac{M_{12}\mathbf{P}(\sigma_1(n+1)=A\mid\sigma_{u_1}=2)}{\mathbf{P}(\sigma_1(n+1)=A)}
\\
&=&\pi_1\left[\frac{M_{11}}{\pi_1}f_n(1,
A)+\frac{M_{12}}{\pi_2}\left(1-f_n(1, A)\right)\right]
\\
&=&\pi_1\left[1+\frac{\theta}{\pi_1}(f_n(1, A)-\pi_1)\right],
\end{eqnarray*}
and thus
\begin{equation}
\mathbf{E}f_{n+1}^*(1, \sigma_1^1(n+1))=\pi_1+\theta^2x_n.
\end{equation}
Therefore it follows from~\eqref{MLE} that
\begin{eqnarray*}
\pi_1+\theta^2x_n\leq\Delta_{n+1}
\leq\pi_1+\pi_1^{1/2}x_{n+1}^{1/2},
\end{eqnarray*}
namely,
\begin{equation}
\label{nftf} x_{n+1}\geq \frac{1}{\pi_1}\theta^4x_n^2\geq
\varrho^4x_n^2.
\end{equation}

Next choosing $\varepsilon=\varrho^2$, it is known by
Lemma~\ref{lemLA} that there exists a $\delta=\delta(\pi,
\varepsilon)>0$ such that if $x_n<\delta$ then
$$
x_{n+1}\geq (d\theta^2-\varepsilon)x_n\geq
(d-1)\varrho^2x_n\geq\varrho^2x_n.
$$
On the other hand, if $x_n\geq \delta$, then~\eqref{nftf} becomes
$x_{n+1}\geq\varrho^4\delta x_n$. Finally taking
$\gamma=\min\{\varrho^2, \varrho^4\delta\}$ completes the proof.
\end{proof}

Actually it seems from~\eqref{explicit} that the estimates of $R$
and $S$ would play a key role in the recursive expression of
$x_{n+1}$. Therefore with the assistant of Lemma~\ref{lemLA} and
Lemma~\ref{ndtf}, resembling (\cite{S}, Corollary 2.14 and Corollary
2.16), we are about to exploit the concentration analysis verifying
that $\frac{\pi_1Z_1}{\pi_1+\pi_2Z_2}$ and $\frac{z_n}{x_n}$ are
both sufficiently around $\pi_1$, and then achieve the proper bounds
of $R$ and $S$. In light of the similar discussion, we skip these
proofs unless it is worth illustrating afresh due to some
qualitative changes caused by the discrepancy between models.
\begin{lemma}
\label{concentration} For any $0<\varepsilon<1$ and $\alpha > 1$
there exist $C=C(\pi, \varepsilon, \alpha)$ and $N_1=N_1(\pi,
\varepsilon, \alpha)$ such that whenever $n\geq N_1$,
$$
P\left(\left|\frac{\pi_1Z_1}{\pi_1Z_1+\pi_2Z_2}-\pi_1\right|>\varepsilon\right)\leq
Cx_n^\alpha.
$$
Moreover, if presume $|\theta|>\varrho$ for some $\varrho>0$ then
there exist $N_2=N_2(\pi, \varepsilon)$ and $\delta=\delta(\pi,
\varepsilon, \varrho)$ such that if $n\geq N_2$ and $x_n\leq\delta$
then
$$
\left|\frac{z_n}{x_n}-\pi_1\right|\leq \varepsilon.
$$
\end{lemma}

With the preceding concentration results, it is feasible to bound
$S$ in~\eqref{S}.
\begin{corollary}
\label{estimateforS} Assume $|\theta|>\varrho$ for some $\varrho>0$.
For any $\varepsilon>0$, there exist $N=N(\pi, \varepsilon)$ and
$\delta=\delta(\pi, \varepsilon, \varrho)>0$ such that if $n\geq N$
and $x_n\leq\delta$ then $|S|\leq\varepsilon x_n^2$.
\end{corollary}
\begin{proof}
For any $\eta>0$, combining Cauchy-Schwartz inequality and
Lemma~\ref{concentration} gives
\begin{eqnarray*}
&|S|&=\left|\mathbf{E}(\pi_1Z_1+\pi_2Z_2-1)^2\left(\frac{\pi_1Z_1}{\pi_1Z_1+\pi_2Z_2}-\pi_1\right)\right|
\\
&\leq&\mathbf{E}\left((\pi_1Z_1+\pi_2Z_2-1)^2\left|\frac{\pi_1Z_1}{\pi_1Z_1+\pi_2Z_2}-\pi_1\right|;\right.
\\
&&\left.\left|\frac{\pi_1Z_1}{\pi_1Z_1+\pi_2Z_2}-\pi_1\right|\leq
\eta\right)
\\
&&+\mathbf{E}\left((\pi_1Z_1+\pi_2Z_2-1)^2\left|\frac{\pi_1Z_1}{\pi_1Z_1+\pi_2Z_2}-\pi_1\right|;\right.
\\
&&\left.\left|\frac{\pi_1Z_1}{\pi_1Z_1+\pi_2Z_2}-\pi_1\right|>
\eta\right)
\\
&\leq&\eta \mathbf{E}(\pi_1Z_1+\pi_2Z_2-1)^2
\\
&&+\mathbf{E}(\pi_1Z_1+\pi_2Z_2-1)^2\mathbf{I}\left(\left|\frac{\pi_1Z_1}{\pi_1Z_1+\pi_2Z_2}-\pi_1\right|>
\eta\right)
\\
&\leq&\eta
\mathbf{E}(\pi_1Z_1+\pi_2Z_2-1)^2+\left[\mathbf{E}(\pi_1Z_1+\pi_2Z_2-1)^4\right]^{1/2}
\\
&&\times\left(\mathbf{P}\left(\left|\frac{\pi_1Z_1}{\pi_1Z_1+\pi_2Z_2}-\pi_1\right|>
\eta\right)\right)^{1/2}.
\end{eqnarray*}
Besides it follows from Lemma~\ref{Taylor} that
$$
\mathbf{E}(\pi_1Z_1+\pi_2Z_2-1)^2\leq C_1(\pi)x_n^2
$$
and
$$
(\mathbf{E}(\pi_1Z_1+\pi_2Z_2-1)^4)^{1/2}\leq C_2(\pi).
$$
Taking $\alpha=6$ in Lemma~\ref{concentration}, there exist
$C_3=C_3(\pi, \eta, \varrho)$ and $N=N(\pi, \eta)$ such that if
$n\geq N$ then
$$
\mathbf{P}\left(\left|\frac{\pi_1Z_1}{\pi_1Z_1+\pi_2Z_2}-\pi_1\right|>\eta\right)\leq
C_3^2x_n^6.
$$
Finally take $\eta=\varepsilon/(2C_1)$ and
$\delta=\varepsilon/(2C_2C_3)$ and thus if $n\geq N$ and
$x_n\leq\delta$ then
\begin{eqnarray*}
|S|\leq\eta C_1x_n^2+C_2C_3x_n^3\leq\varepsilon x_n^2.
\end{eqnarray*}
\end{proof}

\subsection{Proof of Theorem~\ref{reconstruction}}
To accomplish the proof, it suffices to show that when $d\theta^2$
is close enough to $1$, $x_n$ does not converge to $0$. For any
fixed $d$ and $\pi$, by the assumption of $d\theta^2\geq1/2$, say,
$|\theta|\geq(2d)^{-1/2}$ take $\varrho=(2d)^{-1/2}$ in
Lemma~\ref{ndtf} and then get $\gamma=\gamma(\pi, d)>0$. When
$\Delta^2>(1-\theta)^2/3$, namely, $1-6\pi_1\pi_2>0$, by
Lemma~\ref{concentration} and Corollary~\ref{estimateforS}, there
exist $N=N(\pi)$ and $\delta=\delta(\pi, d)>0$ such that if $n\geq
N$ and $x_n\leq\delta$ then the remainders in~\eqref{explicit} could
be bounded respectively by
\begin{equation}
\label{6.1}
|R|\leq\frac{1}{6}\frac{1-6\pi_1\pi_2}{\pi_1\pi_2^2}\frac{d(d-1)}{2}\theta^4x_n^2
\end{equation}
\begin{equation}
\label{6.2}
|S|\leq\frac{1}{6}\frac{1-6\pi_1\pi_2}{\pi_1\pi_2^2}\frac{d(d-1)}{2}\theta^4x_n^2
\end{equation}
\begin{equation}
\label{6.3}
|T|\leq\frac{1}{6}\frac{1-6\pi_1\pi_2}{\pi_1\pi_2^2}\frac{d(d-1)}{2}\theta^4x_n^2
\end{equation}
Consequently combining~\eqref{6.1}, \eqref{6.2} and~\eqref{6.3}
together gives
\begin{equation}
\label{6.4} x_{n+1}\geq
d\theta^2x_n+\frac{1}{2}\frac{(1-6\pi_1\pi_2)}{\pi_1\pi_2^2}\frac{d(d-1)}{2}\theta^4x_n^2.
\end{equation}
Furthermore in light of $x_0=1-\pi_1=\pi_2$ and Lemma~\ref{ndtf},
for all $n$ we have
\begin{equation}
\label{6.5} x_n\geq \pi_2\gamma^n.
\end{equation}
Thus define $\varepsilon=\varepsilon(\pi, d)=\min\{\pi_2\gamma^{N},
\delta\gamma\}>0$, and then~\eqref{6.5} implies that $x_n\geq
\varepsilon$ when $n\leq N$. Next by choosing suitable
$|\theta|<d^{-1/2}$, it is feasible to achieve
\begin{eqnarray}
\label{6.6}
d\theta^2+\frac{1}{2}\frac{(1-6\pi_1\pi_2)}{\pi_1\pi_2^2}\frac{d(d-1)}{2}\theta^4\varepsilon\geq1,
\end{eqnarray}
since $\varepsilon$ is independent of $\theta$. Therefore, suppose
$x_n\geq\varepsilon$ for some $n\geq N$. If $x_n\geq
\gamma^{-1}\varepsilon$, then Lemma~\ref{ndtf} gives $x_{n+1}\geq
\gamma x_n\geq \varepsilon$. If $\varepsilon\leq x_n\leq
\gamma^{-1}\varepsilon\leq\delta$ then by~\eqref{6.4}
and~\eqref{6.6},
\begin{eqnarray*}
x_{n+1}&\geq&
d\theta^2x_n+\frac{1}{2}\frac{(1-6\pi_1\pi_2)}{\pi_1\pi_2^2}\frac{d(d-1)}{2}\theta^4x_n^2
\\
&\geq&x_n\left[d\theta^2+\frac{1}{2}\frac{(1-6\pi_1\pi_2)}{\pi_1\pi_2^2}\frac{d(d-1)}{2}\theta^4\varepsilon\right]
\\
&\geq&x_n \geq\varepsilon.
\end{eqnarray*}
Finally show by induction that $x_n\geq \varepsilon$ for all $n$,
namely, the Kesten-Stigum bound is not tight.

\section{High degree discussion}
\subsection{Gaussian approximation}
For $1\leq j \leq d$, define
\begin{equation}
U_j=\log\left[1+\frac{\theta}{\pi_1}(Y_j-\pi_1)\right]
\end{equation}
and
\begin{equation}
V_j=\log\left[1-\frac{\theta}{\pi_2}(Y_j-\pi_1)\right]
\end{equation}
\begin{lemma}
\label{meansandvariances} There exist positive constants $C=C(\pi)$
and $D=D(\pi)$ such that when $d>D$,
$$\left|d\mathbf{E}U_j-\frac{d\theta^2}{2\pi_1}x_n\right|\leq
Cd^{-1/2};$$
$$\left|d\mathbf{E}V_j+\frac{1+\pi_2}{2\pi_2^2}d\theta^2x_n\right|\leq
Cd^{-1/2};
$$
$$
\left|d\mathbf{Var}(U_j)-\frac{d\theta^2}{\pi_1}x_n\right|\leq
Cd^{-1/2};$$
$$\left|d\mathbf{Var}(V_j)-\frac{\pi_1}{\pi_2^2}d\theta^2x_n\right|\leq
Cd^{-1/2};
$$
$$
\left|d\mathbf{Cov}(U_j, V_j)+\frac{d\theta^2}{\pi_2}x_n\right|\leq
Cd^{-1/2}.
$$
\end{lemma}
\begin{proof}
Starting with the Taylor series expansion of $\log(1+w)$, there
exists a constant $W > 0$ such that when $|w| < W$,
\begin{equation}
\label{log} \left|\log(1 + w)- w + \frac{w^2}{2}\right|\leq|w|^3.
\end{equation}
If we take $D=D(\pi)$ sufficiently large, when $d>D$, $|\theta|\leq
d^{-1/2}$ is small enough to guarantee~\eqref{log} for
$w=\theta(Y_j-\pi_1)/\pi_1$ and then
\begin{eqnarray*}
\left|\mathbf{E}U_j-\frac{\theta^2}{2\pi_1}x_n\right|
&\leq&\mathbf{E}\frac{\theta^3}{\pi_1^3}|Y_j-\pi_1|^3+\frac{\theta^3}{2\pi_1^2}|z_n-\pi_1x_n|
\\
&\leq&\frac{\theta^3}{\pi_1^3}+\frac{\theta^3}{2\pi_1^2}
\\
&\leq&C(\pi)d^{-3/2}
\end{eqnarray*}
for some constant $C=C(\pi)$, where the third inequality follows
from $0\leq z_n\leq x_n\leq 1$. The rest estimates would follow
similarly.
\end{proof}

In view of the complexity of~\eqref{recursion}, it is convenient to
come up with the "better" recursive approximation under results of
Lemma~\ref{meansandvariances}. Define a 2-dimensional vector
\begin{equation}
\mu=(\mu_1, \mu_2)=\left(\frac{1}{2\pi_1},
-\frac{1+\pi_2}{2\pi_2^2}\right)
\end{equation} and a $2\times
2$-covariance matrix
\begin{equation}
\Sigma=\left(
  \begin{array}{cc}
    \frac{1}{\pi_1} & -\frac{1}{\pi_2} \\
    -\frac{1}{\pi_2} & \frac{\pi_1}{\pi_2^2} \\
  \end{array}
\right).
\end{equation}
Suppose $(W_1, W_2)$ has a Gaussian distribution $\mathbf{N}(0,
\Sigma)$, and then $(s\mu_1+\sqrt{s}W_1, s\mu_2+\sqrt{s}W_2)$ is
distributed according to $\mathbf{N}(s\mu, s\Sigma)$. According to
\begin{eqnarray*}
x_{n+1}&=&\mathbf{E}\frac{\pi_1Z_1}{\pi_1Z_1+\pi_2Z_2}-\pi_1
\\
&=&\mathbf{E}\frac{\pi_1\exp\left(\sum_{j=1}^dU_j\right)}{\pi_1\exp\left(\sum_{j=1}^dU_j\right)+\pi_2\exp\left(\sum_{j=1}^dV_j\right)}-\pi_1,
\end{eqnarray*}
we could construct a differentiable function
\begin{eqnarray}
f(s)&=&\mathbf{E}\frac{\pi_1\exp(s\mu_1+\sqrt{s}W_1)}{\pi_1\exp(s\mu_1+\sqrt{s}W_1)+\pi_2\exp(s\mu_2+\sqrt{s}W_2)}\nonumber
\\
&&-\pi_1.
\end{eqnarray}
\begin{lemma}
\label{increasing} The function $f(s)$ is continuously
differentiable and increasing on the interval $(0, \pi_2]$.
\end{lemma}
\begin{proof}
Now let $(W_1', W_2')$ be an independent copy of $(W_1, W_2)$. Thus
if $0\leq s'\leq s$, it is feasible to construct the equivalent
distributions such as
\begin{equation}
\sqrt{s}(W_1,
W_2)\sim\sqrt{s'}(W_1, W_2)+\sqrt{s-s'}(W_1', W_2').
\end{equation}
In view of $(W_1, W_2) \sim \mathbf{N}(0, \Sigma)$, it follows that
$\mathbf{E}(W_2-W_1)=0$ and
\begin{eqnarray*}
\mathbf{Var}(W_2-W_1)^2
&=&\mathbf{E}W_2^2+\mathbf{E}W_1^2-2\mathbf{E}W_1W_2
\\
&=&\frac{1}{\pi_1}+\frac{\pi_1}{\pi_2^2}-2\left(-\frac{1}{\pi_2}\right)
\\
&=&\frac{1}{\pi_1\pi_2^2},
\end{eqnarray*}
which implies that $W_2-W_1$ and $W_2'-W_1'$ are both distributed as
$\mathbf{N}(0, a)$ with $a=1/\pi_1\pi_2^2$.

Next it is well known that if $W$ has the distribution
$\mathbf{N}(\mu, \sigma^2)$, the expectation of the exponential
random variable could be estimated as
\begin{equation}
\label{exponential} \mathbf{E}e^W=e^{\mu+\frac{\sigma^2}{2}},
\end{equation}
based on which, we are allowed to estimate the conditional
expectation given $W_1$ and $W_2$:
\begin{equation*}
\mathbf{E}\left[\exp(\sqrt{s'}(W_2-W_1)+\sqrt{s-s'}(W_2'-W_1'))\mid\{W_1,
W_2\}\right]
\end{equation*}
\begin{equation}
=\exp\left[\sqrt{s'}(W_i-W_1)+\frac{a}{2}(s-s')\right].
\end{equation}
Then apply Jensen's inequality, plus noting that the
function $(1+x)^{-1}$ is convex and
$\mu_2-\mu_1=-(1+\pi_2)/(2\pi_2^2)-1/(2\pi_1)=-1/(2\pi_1\pi_2^2)=-a/2$,
to achieve
$$
f(s) \geq f(s').
$$
Next given $s>0$, it is concluded that
\begin{equation*}
\begin{aligned}
&\mathbf{E}\left|\frac{\partial}{\partial
s}\frac{\pi_1\exp(s\mu_1+\sqrt{s}W_1)}{\pi_1\exp(s\mu_1+\sqrt{s}W_1)+\pi_2\exp(s\mu_2+\sqrt{s}W_2)}\right|
\\
&=\mathbf{E}\left|\frac{\partial}{\partial
s}\left(1+\frac{\pi_2}{\pi_1}\exp(s(\mu_2-\mu_1)+\sqrt{s}(W_2-W_1))\right)^{-1}\right|
\\
&\leq\frac14\mathbf{E}\left|\mu_2-\mu_1+\frac{W_2-W_1}{2\sqrt{s}}\right|
\\
&<\infty,
\end{aligned}
\end{equation*}
from the fact that
$\left|\frac{\pi_2}{\pi_1}e^t/\left(1+\frac{\pi_2}{\pi_1}e^t\right)^2\right|\leq1/4$
holds for any $t\in\mathbb{R}$. Then we establish the
differentiability of $f(s)$.
\end{proof}
Now we can reinvestigate the recursive approximation with the
assistance of $f(s)$ and the following lemma established immediately
by using Central Limit Theorem, Gaussian approximation and
Portmanteau Theorem.
\begin{lemma}
\label{Gaussianapproximation} For arbitrary $\varepsilon>0$ there
exists a $D=D(\pi, \varepsilon)>0$ such that whenever $d>D$,
$$
\big{|}x_{n+1}-f(d\theta^2x_n)\big{|}\leq \varepsilon.
$$
\end{lemma}

\subsection{Proof of Theorem~\ref{nonreconstruction}}
First referring to Mathematica, it is possible to establish
\begin{lemma}
\label{7.7} When $\Delta^2<(1-\theta)^2/3$, for any $0< s \leq
\pi_2$ we have
$$
f(s) < s.
$$
\end{lemma}
When $\Delta^2<(1-\theta)^2/3$, namely, $1-6\pi_1\pi_2<0$, the proof
of Theorem~\ref{nonreconstruction} would resemble
Theorem~\ref{reconstruction} to establish the analogous recursive
inequality of~\eqref{6.4} under the condition of $x_n\leq\delta$ and
$n\geq N$ for suitable $\delta=\delta(\pi, d)$ and $N=N(\pi)$.
However, there still exists a crucial discrepancy between two
proofs, that is, Theorem~\ref{nonreconstruction} relies tightly on
large $d$, while the case of small degree emerges as the more
intractable issue. Thus in order to establish an analogue
of~\eqref{6.4}, it is necessary to eliminate the dependency of
$\delta$ on $d$ as following.

{\bf Proof of Theorem~\ref{nonreconstruction}.} Resembling the proof
of Theorem~\ref{reconstruction}, we evaluate $R$, $S$ and $T$
in~\eqref{explicit} respectively, under $1-6\pi_1\pi_2<0$. First
take $D=D(\pi)=(6C_R(\pi)\pi_1\pi_2^2)^2/(6\pi_1\pi_2-1)^2$ such
that if $d>D$ that implies $|\theta|\leq d^{-1/2}\leq D^{-1/2}$,
then it is concluded by recalling~\eqref{R} and
$\left|z_n/x_n-\pi_1\right|\leq1$ that
\begin{eqnarray*}
|R|&\leq&
C_R(\pi)\frac{d(d-1)}{2}|\theta|^5\left|\frac{z_n}{x_n}-\pi_1\right|x_n^2
\end{eqnarray*}
\begin{eqnarray}
\label{7.1R}
&\leq&-\frac{1}{6}\frac{1-6\pi_1\pi_2}{\pi_1\pi_2^2}\frac{d(d-1)}{2}\theta^4x_n^2.
\end{eqnarray}

Next applying Lemma~\ref{concentration} to deal with $S$, there
exist $N=N(\pi)$ and $\delta=\delta(\pi)>0$ independent of $d$ such
that if $n\geq N$ and $x_n<\delta$ then the analogues of \eqref{6.2}
and \eqref{6.3} still hold as
\begin{equation}
\label{7.2S}
|S|\leq-\frac{1}{6}\frac{1-6\pi_1\pi_2}{\pi_1\pi_2^2}\frac{d(d-1)}{2}\theta^4x_n^2
\end{equation}
and
\begin{equation}
\label{7.3}
|T|\leq-\frac{1}{6}\frac{1-6\pi_1\pi_2}{\pi_1\pi_2^2}\frac{d(d-1)}{2}\theta^4x_n^2.
\end{equation}
Finally taken together, if $d>D$, $n\geq N$ and $x_n<\delta$,
then~\eqref{7.1R}, \eqref{7.2S} and~\eqref{7.3} give
\begin{equation}
\label{7.2} x_{n+1}\leq
d\theta^2x_n+\frac{1}{2}\frac{(1-6\pi_1\pi_2)}{\pi_1\pi_2^2}\frac{d(d-1)}{2}\theta^4x_n^2
\leq x_n.
\end{equation}
Therefore it follows that $L=\lim_{n\to\infty}x_n$ does exist, since
the sequence $\{x_n\}_{n\geq N}$ is bounded and decreasing. Thus
taking limits to both sides of~\eqref{7.2} yields
\begin{equation}
L\leq
d\theta^2L-\frac{1}{2}\frac{(6\pi_1\pi_2-1)}{\pi_1\pi_2^2}\frac{d(d-1)}{2}\theta^4L^2,
\end{equation}
which implies $L=\lim_{n\rightarrow \infty}x_n=0$, and hence
non-reconstruction.

So here it suffices to find some $m\geq N$ such that $x_m<\delta$.
Define $\varepsilon=\varepsilon(\pi,
\delta)=\varepsilon(\pi)=\frac{1}{2}\min_{s\geq \delta}(s-f(s))$.
Since the function $s-f(s)$ is continuous and positive on $[\delta,
\pi_2]$, it follows by Lemma~\ref{7.7} that $\varepsilon>0$. Then by
Lemma~\ref{Gaussianapproximation}, there exists a $D=D(\pi,
\varepsilon)=D(\pi)>0$ such that when $d>D$, if $x_n\geq\delta$ and
$n\geq N$, we have
\begin{eqnarray*}
x_{n+1}<f(d\theta^2x_n)+\varepsilon \leq f(x_n)+\varepsilon\leq
x_n-\varepsilon,
\end{eqnarray*}
where the second inequality is from Lemma~\ref{increasing}, say,
$f(s)$ is increasing on $[0, \pi_2]$. Therefore there must exist
$m\geq N$ such that $x_m<\delta$, as desired.

\section{Conclusion}

In this paper, we have studied the asymmetric Ising model on regular
trees and figured out the critical conditions for the reconstruction
by means of the recursive structure of the tree. The key idea is to
analyze the relation between the distributions
$\mathbf{P}(\sigma_\rho=1\mid\sigma^1(n))$ and
$\mathbf{P}(\sigma_\rho=1\mid\sigma^1(n+1))$.

Our result not only establishes the existence the symmetry bias to
keep Kesten-Stigum bound tight, but determines the exact thresholds
for non-solvability of the asymmetric Ising model.

More importantly, together with some results of nonlinear dynamic
system, our skills could also be applied to explore the
reconstruction of the $d$-ary tree on continuous state space, and
even the general phylogenetic reconstruction.



\end{multicols}

\end{document}